\documentclass{amsart}
\usepackage{a4wide,amssymb,hyperref}

\newcommand{\IR}{\mathbb R}

\newcommand{\II}{\mathbb I}
\newcommand{\IQ}{\mathbb Q}
\newcommand{\w}{\omega}
\newcommand{\e}{\varepsilon}
\newcommand{\pr}{\mathrm{pr}}
\newcommand{\M}{\mathcal M}
\newcommand{\Ra}{\Rightarrow}

\newtheorem{theorem}{Theorem}

\newtheorem{problem}{Problem}
\newtheorem{lemma}{Lemma}
\newtheorem{corollary}{Corollary}
\newtheorem{claim}{Claim}
\theoremstyle{definition}
\newtheorem{remark}{Remark}

\title[Selection properties of the split interval]{Selection properties of the split interval\\ and the Continuum Hypothesis}
\author{Taras Banakh} 
\address{Ivan Franko National University of Lviv (Ukraine) and Jan Kochanowski University in Kielce (Poland)}
\email{t.o.banakh@gmail.com}
\keywords{Continuum Hypothesis, split interval, measurable selection, Borel selection, usco multimap, fragmentable compact, Rosenthal compact, GO-space}
\subjclass{54C65, 54F05, 54H05, 03E50}
\begin{document} 
\begin{abstract}
We prove that every usco multimap $\Phi:X\to Y$ from a metrizable separable space $X$ to a GO-space $Y$ has an $F_\sigma$-measurable selection. On the other hand, for the split interval $\ddot\II$ and the projection $P:\ddot\II^2\to\II^2$ of its square onto the unit square $\II^2$, the usco multimap $P^{-1}:\II^2\multimap\ddot\II^2$ has a Borel ($F_\sigma$-measurable) selection if and only if the Continuum Hypothesis holds. This CH-example shows that know results on Borel selections of usco maps into fragmentable compact spaces cannot be extended to a wider class of compact spaces.
\end{abstract}
\maketitle

\section{Introduction}
By a {\em multimap} $\Phi:X\multimap Y$ between topological spaces $X,Y$ we understand any subset $\Phi\subset X\times Y$, which can be thought as a function   assigning to every point $x\in X$ the subset $\Phi(x):=\{y\in Y:\langle x,y\rangle\in\Phi\}$ of $Y$. For a subset $A\subset X$ we put $\Phi[A]=\bigcup_{x\in A}\Phi(x)$. Each function $f:X\to Y$ can be thought as a single-valued multimap $\{\langle x,f(x)\rangle:x\in X\}\subset X\times Y$.

For a multimap $\Phi:X\multimap Y$, its {\em inverse multimap} $\Phi^{-1}:Y\multimap X$ is defined by $\Phi^{-1}:=\{\langle y,x\rangle:\langle x,y\rangle\in\Phi\}$.

A multimap $\Phi:X\multimap Y$ is called 
\begin{itemize}
\item {\em lower semicontinuous} if  for any open set $U\subset Y$ the set $\Phi^{-1}[U]$ is open in $X$;
\item {\em upper semicontinuous} if for any closed set $F\subset Y$ the set $\Phi^{-1}[F]$ is closed in $X$;
\item {\em Borel-measurable} if for any Borel set $B\subset Y$ the set $\Phi^{-1}[B]$ is Borel in $X$;
\item {\em compact-valued} if for every $x\in X$ the subspace $\Phi(x)$ of $Y$ is compact and non-empty;
\item {\em usco} if $\Phi$ is upper semicontinuous and compact-valued.
\end{itemize}
It is well-known that for any surjective continuous function $f:X\to Y$ between compact Hausdorff spaces, the inverse multimap $f^{-1}:Y\multimap X$ is usco.

Let $\Phi:X\multimap Y$ be a multimap between topological spaces. 
A function $f:X\to Y$ is called a {\em selection} of $\Phi$ if $f(x)\in\Phi(x)$ for every $x\in X$. The Axiom of Choice ensures that every multimap $\Phi:X\multimap Y$ with non-empty values has a selection. The problem is to find selections possessing some additional properties like the continuity or measurability.

One of classical results in this direction is the following theorem of Kuratowski and Ryll-Nardzewski \cite{KRN} (see also \cite[\S5.2]{Sri} or \cite[6.12]{RS}).

\begin{theorem} Let $X,Y$ be Polish spaces.  Any Borel-measurable multimap $\Phi:X\multimap Y$ with non-empty values has a Borel-measurable selection.     
\end{theorem}

We recall that a function $f:X\to Y$ between topological spaces is called {\em Borel-measurable} (resp. {\em $F_\sigma$-measurable}) if for every open set $U\subset Y$ the preimage $f^{-1}[U]$ is Borel (or type $F_\sigma$) in $X$. 

$F_\sigma$-Measurable selections of usco multimaps with values in non-metrizable compact spaces were studied by many mathematicians \cite{JR1}, \cite{JR2}, \cite{JR3}, \cite{Hansell}, \cite{HJT}, \cite{JOPV}. Positive results are known for two  classes of compact spaces: fragmentable and linearly ordered.

Let us recall \cite[5.0.1]{Fab} (see also \cite[\S6]{Shah}) that a topological space $K$ is {\em fragmentable} if $K$ has a metric $\rho$ such that for every $\varepsilon>0$  each non-empty subset $A\subset K$ contains a non-empty relatively open set $U\subset A$ of $\rho$-diameter $<\e$. By \cite[5.1.12]{Fab}, each fragmentable compact Hausdorff space contains a metrizable dense $G_\delta$-subspace.

The following selection theorem can be deduced from Theorem 1' and Lemma 6 in  \cite{HJT}.

\begin{theorem}[Hansell, Jayne, Talagrand]\label{t:HJT} Any usco map $\Phi:X\to K$ from a perfectly paracompact space $X$ to a fragmentable compact space $Y$ has an $F_\sigma$-measurable selection.
\end{theorem} 

A similar selection theorem holds for usco maps into countably cellular GO-spaces.
A Hausdorff topological space $X$ is called a {\em generalized ordered space} (briefly, a {\em GO-space}) if $X$ admits a linear order $\le$ such that the topology of $X$ is generated by a base consisting of open order-convex subsets of $X$. A subset $C$ of a linearly ordered space $X$ is called {\em order-convex} if for any points $x\le y$ in $C$ the order interval $[x,y]:=\{z\in X:x\le z\le y\}$ is contained in $X$. We say that the topology of $X$ is generated by the linear order $\le$ if the topology of $X$ is generated by the subbase $\{(\leftarrow,a),(a,\to):a\in X\}$ consisting the the order-convex sets $(\leftarrow,a):=\{x\in X:x<a\}$ and $(a,\to)=\{x\in X:a<x\}$.

A topological space $X$ is {\em countably cellular} if every disjoint family of open sets in $X$ is at most countable. It is easy to see that each separable topological space is countably cellular. A topological space is called {\em $F_\sigma$-perfect} if every open set in $X$ is of type $F_\sigma$ in $X$ (i.e., can be represented as the countable union of closed sets). For example, every metruzable space is $F_\sigma$-perfect.

The following  selection theorem will be proved in Section~\ref{s:L}.

\begin{theorem}\label{t:L} Let $Y$ be a GO-space and $X$ be an ($F_\sigma$-perfect) topological space. If $X$ or $Y$ is countably cellular, then any usco map $\Phi:X\multimap Y$  has a Borel ($F_\sigma$-measurable) selection.
\end{theorem}

Theorems~\ref{t:HJT}, \ref{t:L} suggest the following problem.

\begin{problem} Is it true that any usco map $\Phi:M\multimap K$ from a compact metrizable  space $M$ to a  compact Hausdorff space $K$  has a Borel ($F_\sigma$-measurable) selection?
\end{problem}

In this paper we prove that this problem has negative answer under the negation of the Continuum Hypothesis (i.e., under $\w_1<\mathfrak c$). A suitable counterexample will be constructed using the split square $\ddot\II$, which is the square of the split interval $\ddot\II$. 

The {\em split interval} is the linearly ordered space $\ddot\II=[0,1]\times\{0,1\}$ whose topology is generated by the lexicographic order (defined by $\langle x,i\rangle \le\langle y,j\rangle$ iff either $x<y$ or else $x=y$ and $i<j$). The split interval plays a fundamental role in the theory of separable Rosenthal compacta \cite{Tod}. Let us recall that a topological space is called {\em Rosenthal compact}  if it is homeomorphic to a compact subspace of the space $B_1(P)$ of functions of the first Baire class on a Polish space $P$. It is well-known (and easy to see) that the split interval is Rosenthal compact and so is its square. By Theorem 4 of Todor\v cevi\'c \cite{Tod}, each non-metrizable Rosenthal compact space of countable spread contains a topological copy of the split interval. A topological space has {\em countable spread} if it contains no uncountable discrete subspaces.


By Theorem~\ref{t:L}, any usco map $\Phi:X\multimap\ddot\II$ from an $F_\sigma$-perfect topological space $X$ has an $F_\sigma$-measurable selection. 
In contrast, the split square $\ddot\II^2$ has dramatically different selections properties.
Let $p:\ddot\II\to\II$, $p:\langle x,i\rangle\mapsto x$, be the natural projection of the split interval onto the unit interval $\II=[0,1]$, and $$P:\ddot\II^2\to\II^2,\;\;P:\langle x,y\rangle\mapsto \langle p(x),p(y)\rangle,$$be the  projection of the split square $\ddot\II^2$ onto the unit square $\II^2$.

\begin{theorem}\label{t:main} The following conditions are equivalent:
\begin{enumerate}
\item the usco multimap $P^{-1}:\II^2\to \ddot\II^2$ has a Borel-measurable selection;
\item the usco multimap $P^{-1}:\II^2\to \ddot\II^2$ has an $F_\sigma$-measurable selection;
\item $\w_1=\mathfrak c$.
\end{enumerate}
\end{theorem}

The implication $(2)\Ra(1)$ of Theorem~\ref{t:main} is trivial and the implications $(1)\Ra(3)\Ra(2)$ are proved in Lemmas~\ref{l:L1} and \ref{l:main}, respectively.

Combining Theorem~\ref{t:main} with the Todor\v cevi\'c dichotomy for Rosenthal compact spaces, we obtain the following consistent characterization of metrizable compacta.

\begin{corollary} Under $\w_1<\mathfrak c$ a Rosenthal compact space $K$ is metrizable if and only if $K$ has countable spread and each usco multimap $\Phi:\II^2\to K^2$ has a Borel-measurable selection.
\end{corollary}

\begin{proof} The ``only if'' part follows from Theorem~\ref{t:HJT}. To prove the ``if'' part, assume that a Rosenthal compact $K$ is not metrizable but has countable spread. By Theorem~4 of \cite{Tod}, the space $K$ contains a topological copy of the split interval $\ddot\II$. We lose no generality assuming that $\ddot\II\subset K$. 
By Theorem~\ref{t:main}, under $\w_1<\mathfrak c$, the usco multimap $P^{-1}:\II^2\multimap \ddot\II^2\subset K^2$ does not have Borel-measurable selections.
\end{proof}

Now we pose some open problems suggested by Theorem~\ref{t:main}.

\begin{problem} Assume CH. Is it true that each usco map $\Phi:X\to\ddot\II^2$ from a metrizable (separable) space $X$ has a Borel-measurable selection?
\end{problem}

Observe that the map $p:\ddot\II\to \II$ is 2-to-1 and its square $P:\ddot\II^2\to\II^2$ is 4-to-1. A function $f:X\to Y$ is called {\em $n$-to-1} for some $n\in\mathbb N$ if $|f^{-1}(y)|\le n$ for any $y\in Y$. By Theorem 3 of Todor\v cevi\'c \cite{Tod}, every Rosenthal compact space of countable spread admits a 2-to-1 map onto a metrizable compact space. Let us observe that the splitted square $\ddot\II^2$ contains a discrete subspace of cardinality continuum and hence has uncountable spread.

\begin{problem} Let $n\in\{2,3\}$. Is there an $n$-to-1 map $f:K\to M$ from a (Rosenthal) compact space $K$ to a metrizable compact space $M$ such that the inverse multimap $f^{-1}:M\multimap K$ has no Borel selections?
\end{problem}

\begin{remark} Theorem~\ref{t:main} provides  a consistent counterexample to the problem \cite{MO} of Chris Heunen, posed on  {\tt Mathoverflow}. 
\end{remark}

\section{Proof of Theorem~\ref{t:L}}\label{s:L}

Theorem~\ref{t:L} follows from Lemmas~\ref{l:L1} and \ref{l:L2}, proved in this section.

First we prove one lemma, showing that our definition of a GO-space agrees with the original definition of Lutzer \cite{Lutzer}. Probably this lemma is known but we could not find the precise reference in the literature.

\begin{lemma}\label{l:GO} The linear order $\le$ of any GO-space $X$ is a closed subset of the square $X\times X$.
\end{lemma}

\begin{proof} Given two elements $x,y\in X$ with $x\not\le y$, use the Hausdorff property of $X$ and find two disjoint order-convex neighborhoods $O_x,O_y\subset X$ of the points $x,y$, respectively. We claim that the product $O_x\times O_y$ is disjoint with the linear order $\le$. Assuming that this is not true, find elements $x'\in O_x$ and $y'\in O_y$ such that $x'\le y'$. Taking into account that the sets  $O_x,O_y$ are disjoint and order-convex, we conclude that $x'<y$ and $x<y'$. It follows from $x\not\le y$ that $y<x$. Then $x'<y<x<y'$, which contradicts the assumption. This contradiction shows that the neighborhood $O_x\times O_y$ of the pair $\langle x,y\rangle$ is disjoint with $\le$ and hence $\le$ is a closed subset of $X\times X$.
\end{proof}

\begin{lemma}\label{l:L1} Any usco multimap  $\Phi:X\multimap Y$  from an ($F_\sigma$-perfect) topological space $X$ to a countably cellular GO-space space $Y$ has a Borel ($F_\sigma$-measurable) selection.
\end{lemma}

\begin{proof} Being a GO-space, $Y$ has a base of the topology consisting of open order-convex subsets with respect to some linear order $\le$ on $Y$. By Lemma~\ref{l:GO},  the linear order $\le$ is a closed subset of $Y\times Y$. Then for every $a\in Y$ the order-convex set $(\leftarrow, a]=\{y\in Y:y\le a\}$ is closed in $Y$, which implies that each non-empty compact subset of $Y$ has the smallest element.

Then for any usco multmap $\Phi:X\multimap Y$ we can define a selection $f:X\to Y$ of $\Phi$ assigning to each point $x\in X$ the smallest element $f(x)$ of the non-empty compact set $\Phi(x)\subset Y$. We claim that this selection is $F_\sigma$-measurable. 

A subset $U\subset Y$ is called {\em upper} if for any $u\in U$ the order-convex set $[u,\to)=\{y\in Y:u\le y\}$ is contained in $U$.

\begin{claim}\label{cl:l1} For any upper open set $U\subset Y$ the preimage $f^{-1}[U]$ is open in $X$.
\end{claim}

\begin{proof} For any $x\in f^{-1}[U]$ we get $\Phi(x)\subset [f(x),\to)\subset U$. The upper semicontinuity of $\Phi$ yields a neighborhood $O_x\subset X$ such that $\Phi[O_x]\subset U$. Consequently, $f(O_x)\subset \Phi[O_x]\subset U$, witnessing that the set $f^{-1}[U]$ is open in $X$. 
\end{proof}

A subset $L\subset Y$ is {\em lower} if for every $a\in L$ the order-convex set $(\leftarrow,a]=\{y\in Y:y\le a\}$ is contained in $L$.

\begin{claim}\label{cl:l2} For any closed lower set $L\subset Y$ the preimage $f^{-1}[L]$ is closed in $X$.
\end{claim}

\begin{proof} Observe the the complement $X\setminus L$ is an open upper set in $Y$. By Claim~\ref{cl:l1}, the preimage $f^{-1}[X\setminus L]$ is open in $X$ and hence its complement $X\setminus f^{-1}[X\setminus L]=f^{-1}[L]$ is closed in $X$.
\end{proof}

\begin{claim}\label{cl:l3} For any lower set $L\subset Y$ the preimage $f^{-1}(L)$ is of type $F_\sigma$ in $X$.
\end{claim}

\begin{proof} If $L$ has the largest element $\lambda$, then $L=(\leftarrow,\lambda]$ and $f^{-1}[L]=f^{-1}[(\leftarrow,\lambda]]$ is closed by Claim~\ref{cl:l2}. So, we assume that $L$ does not have the largest element. Then the countable cellularity of $Y$ implies that $L$ has a countable cofinal subset $C\subset L$ (which means that for every $x\in L$ there exists $y\in C$ with $x\le y$). By Lemma~\ref{cl:l2}, for every $c\in C$ the preimage $f^{-1}[(\leftarrow,c]]$ is closed in $X$. Since $L=\bigcup_{c\in C}(\leftarrow, c]$, the preimage $f^{-1}[L]=\bigcup_{c\in C}f^{-1}[(\leftarrow,c]]$ is of type $F_\sigma$ in $X$.
\end{proof} 

\begin{claim}\label{cl:l4} For any open order-convex subset $U\subset Y$ the preimage $f^{-1}[U]$ is a Borel subset of $X$ (of type $F_\sigma$ if the space $X$ is $F_\sigma$-perfect). 
\end{claim}

\begin{proof} The order-convexity of $U$ implies that $U=\overleftarrow U\cap \overrightarrow U$ where $\overleftarrow U=\bigcup_{u\in U}(\leftarrow,u]$ and $\overrightarrow U=\bigcup_{u\in U}[u,\to)$. Taking into account that $Y$ has a base of order-convex sets, one can show that the upper set $\overrightarrow U$ is open in $X$. By Claim~\ref{cl:l1}, the preimage $f^{-1}[\overrightarrow U]$ is open in $X$
(of type $F_\sigma$ if the space $X$ is $F_\sigma$-perfect). By Claim~\ref{cl:l3}, the preimage $f^{-1}[\overleftarrow U]$ is of type $F_\sigma$ in $X$. Then $f^{-1}(U)=f^{-1}[\overleftarrow U]\cap f^{-1}[\overrightarrow U]$ is Borel (of type $F_\sigma$ if $X$ is $F_\sigma$-perfect).
\end{proof}

\begin{claim}\label{cl:l5} For every open set $U\subset Y$ the preimage $f^{-1}[B]$ is Borel subset of $X$ (of type $F_\sigma$ if $X$ is $F_\sigma$-perfect).
\end{claim}

\begin{proof}  By the definition of the topology of $Y$, each point $x\in U$ has an open order-convex neighborhood $O_x\subset U$. By the Kuratowski-Zorn Lemma, each open order-convex subset of $U$ is contained in a maximal open order convex subset of $U$. Let $\mathcal C\subset\mathcal B$ be the family of maximal open order-convex subsets of $U$. Observe that $U=\bigcup\mathcal C$ and any distinct sets $C,D\in\mathcal C$ are disjoint: otherwise the union $C\cup D$ would be an open order convex subset of $U$ and by the maximality of $C$ and $D$, $C=C\cup D=D$. Since the space $Y$ is countably cellular, the family $\mathcal C$ is at most countable. By Claim~\ref{cl:l4}, for every $C\in\mathcal C$ the preimage $f^{-1}(C)$ is Borel (an type $F_\sigma$-set if $X$ is $F_\sigma$-perfect) and so is  the countable union $f^{-1}[U]=\bigcup_{C\in\mathcal C}f^{-1}$.
\end{proof}

Claim~\ref{cl:l5} completes the proof of Lemma~\ref{l:L1}.
\end{proof}

\begin{lemma}\label{l:L2} Every usco multimap $\Phi:X\multimap Y$ from a countably cellular ($F_\sigma$-perfect) topological space $X$ into a GO-space $Y$ has a Borel ($F_\sigma$-measurable) selection.
\end{lemma}

\begin{proof} The Kuratowski-Zorn Lemma implies that the usco map $\Phi$ contains a minimal usco map $\Psi:X\multimap Y$. We claim that the image $\Psi[X]\subset Y$ is a countably cellular subspace of $Y$. Assuming the opposite, we can find an uncountable disjoint family $(U_\alpha)_{\alpha\in\w_1}$ of non-empty open subsets in $\Psi[X]$. For every $\alpha\in\w_1$, find $x_\alpha\in X$ such that $\Phi(x_\alpha)\cap U_\alpha\ne\emptyset$. By Lemma 3.1.2 \cite{Fab}, the minimality of the usco map $\Psi$ implies that $\Psi[V_\alpha]\subset U_\alpha$ for some non-empty open set $V_\alpha\subset X$. Taking into account that the family $(U_\alpha)_{\alpha\in\w_1}$ is disjoint, we conclude that the family $(V_\alpha)_{\alpha\in\w_1}$ is disjoint, witnessing that the space $X$ is not countably cellular. But this contradicts our assumption. This contradiction shows that the GO-subspace $\Psi[X]$ of $Y$ is countably cellular. By Lemma~\ref{l:L1}, the usco map $\Psi:X\to\Psi[X]$ has a Borel ($F_\sigma$-measurable) selection, which is also a selection of the usco map $\Phi$.
\end{proof}

Finally, let us prove one selection property of the split interval, which will be used  in the proof of Lemma~\ref{l:main}.

\begin{lemma}\label{l:I} Any selection of the multimap $p^{-1}:\II\to\ddot\II$ is $F_\sigma$-measurable.
\end{lemma} 

\begin{proof} Given any open subset $U\subset \ddot\II$, we need to show that $s^{-1}[U]$ is of type $F_\sigma$ in $\II$. For every $x\in s^{-1}[U]$, find an open order-convex set $I_x\subset U$ containing $s(x)$. It is well-known (see e.g. \cite[3.10.C(a)]{Eng}) that the split interval $\ddot\II$ is hereditarily Lindel\"of. Consequently, there exists a countable set $C\in s^{-1}[U]$ such that $\bigcup_{x\in s^{-1}[U]}I_x=\bigcup_{x\in C}I_x$ and hence $s^{-1}[U]=\bigcup_{x\in C}s^{-1}[I_x]$. For every $x\in C$ the order-convexity of the interval $I_x\subset \ddot\II$ implies that its preimage $s^{-1}[I_x]$ is a convex subset of $\II$, containing $x$. Since convex subsets of $\II$ are of type $F_\sigma$, the countable union $s^{-1}[U]=\bigcup_{x\in C}I_x$ is an $F_\sigma$-set in $\II$. 
\end{proof}

\section{Selection properties of the split square $\ddot\II^2$ under the negation of CH}

In this section we study the selection properties of the split square $\ddot\II^2$ under the negation of the Continuum Hypothesis.

By $\langle x,y\rangle$ we denote ordered pairs of elements $x,y$. In this way we distinguish ordered pairs from the order intervals $(x,y):=\{z\in x\le  z\le y\}$ in linearly ordered spaces.

The split interval $\ddot\II=\II\times\{0,1\}$ carries the lexicographic order defined by $\langle x,i\rangle\le\langle y,j\rangle$ iff either $x<y$ or ($x=y$ and $i<j$). It is well-known that the topology generated by the lexicographic order on $\ddot\II$ is compact and Hausdorff, see \cite[3.10.C(b)]{Eng}. By $p:\ddot\II\to\II$, $p:\langle x,i\rangle\mapsto x$, we denote the coordinate projection and by $P:\ddot\II^2\to\II^2$, $P:\langle x,y\rangle\mapsto \langle p(x),p(y)\rangle$ the square of the map $p$.



The following lemma proves the implication $(1)\Ra(3)$ of Theorem~\ref{t:main}.

\begin{lemma}\label{l1} If $\w_1<\mathfrak c$, then the multimap $P^{-1}:\II^2\to \ddot\II^2$ has no Borel selections.
\end{lemma}

\begin{proof} To derive a contradiction, assume that the multimap $P^{-1}$ has a Borel-measurable selection $s:\II^2\to \ddot\II^2$.

For a real number $x\in\II$ by $x_0$ and $x_1$ we denote the points $\langle x,0\rangle$ and $\langle x,1\rangle$ of the split interval $\ddot\II$. Then $\ddot\II=\ddot\II_0\cup \ddot\II_1$ where $\ddot\II_i=\{x_i:x\in\II\}$ for $i\in\{0,1\}$.

For any numbers $i,j\in\{0,1\}$ consider the set $$Z_{ij}=\{z\in \II^2:s(z)\in \ddot\II_i\times \ddot\II_j\}$$and observe that $\II^2=\bigcup_{i,j=0}^1Z_{ij}$.

For a point $a\in \ddot\II$, let $[0_0,a)$ and $(a,1_1]$ be the order intervals in $\ddot\II$ with respect to the lexicographic order. Observe that for any $x\in\II$ we have
$$p\big([0_0,x_0)\big)=[0,x),\;\;p\big([0_0,x_1)\big)=[0,x],\;\;p\big((x_0,1_1]\big)=[x,1],\;\;p\big((x_1,1_1]\big)=(x,1].$$

For every $a\in(0,2)\subset\IR$ consider the lines $$\mathrm{L}_a=\{\langle x,y\rangle\in\IR^2:x+y=a\}\mbox{ \ and \ }\Gamma^a=\{\langle x,y\rangle\in\IR^2:y-x=a\}$$ on the plane. 

\begin{claim}\label{cl1} For every $a\in \IR$ the intersection ${\mathrm L}_a\cap Z_{00}$ is at most countable.
\end{claim}

\begin{proof} If for some $a\in\IR$ the intersection ${\mathrm L}_a\cap Z_{00}$ is uncountable, then we can choose a non-Borel subset $B\subset \mathrm L_a\cap Z_{00}$ of cardinality $|B|=\w_1$. For every point $\langle x,y\rangle\in B\subset Z_{00}$, the definition of the set $Z_{00}$ ensures that $s(\langle x,y\rangle)=\langle x_0,y_0\rangle$ and hence the set $U_{\langle x,y\rangle}=[0_0, x_1)\times[0_0, y_1)=[0_0,x_0]\times[0_0,y_0]$ is an open neighborhood of $s(\langle x,y\rangle)$ in $\ddot\II^2$. Observe that $\langle x,y\rangle\in s^{-1}(U_{\langle x,y\rangle})\subset p(U_{\langle x,y\rangle})=[0,x]\times[0,y]$ and hence $\mathrm{L}_a\cap s^{-1}(U_{\langle x,y\rangle})=\{\langle x,y\rangle\}$. Then for the open set $U=\bigcup_{\langle x,y\rangle\in B}U_{\langle x,y\rangle}$ the preimage $s^{-1}[U]$ is not Borel in $\II^2$ because the intersection $s^{-1}[U]\cap\mathrm L_a=B$ is not Borel. But this contradicts the Borel measurability of $s$.
\end{proof}

By analogy we can prove the following claims.

\begin{claim}\label{cl2} For every $a\in \IR$ the intersection $\mathrm L_a\cap Z_{11}$ is at most countable.
\end{claim}

\begin{claim}\label{cl3} For every $b\in \IR$ the intersection $\Gamma^b\cap (Z_{01}\cup Z_{10})$ is at most countable.
\end{claim}

Now fix any subset set $\Omega\subset [\frac12,\frac32]$ of cardinality $|\Omega|=\w_1$. By Claims~\ref{cl1}, \ref{cl2}, for every $a\in \Omega$ the intersection $\mathrm{L}_a\cap (Z_{00}\cup Z_{11})$ is at most countable. Consequently the union $$U=\bigcup_{a\in\Omega}\mathrm{L}_a\cap(Z_{00}\cup Z_{11})$$ has cardinality $|U|\le\w_1$. Since $|U|\le\w_1<\mathfrak c$, there exists a real number $b\in [\frac12,\frac32]$ such that the line $\Gamma^b$ does not intersect the set $U$. Since $\{b\}\cup\Omega\subset [\frac12,\frac32]$ for every $a\in\Omega$ the intersection $\Gamma^b\cap \mathrm{L}_a\cap\II^2$ is not empty. Then the set $X=\bigcup_{a\in\Omega}\mathrm{L}_a\cap\Gamma^b\subset\II^2$ is uncountable and $X\subset \Gamma^b\setminus U\subset\Gamma^b\cap (Z_{01}\cup Z_{10})$. But this contradicts Claim~\ref{cl3}.
\end{proof}

\section{Selection properties of the split square $\ddot\II^2$ under the Continuum Hypothesis}\label{s:proofs}

In this section we shall prove that under the continuum hypothesis the usco multimap $P^{-1}:\II^2\to\ddot\II^2$ has an $F_\sigma$-measurable selection.

First we introduce some terminology related to monotone functions.

A subset $f\subset \II^2$ is called a 
\begin{enumerate}
\item a {\em function} if for any $\langle x_1,y_1\rangle,\langle x_2,y_2\rangle\in f$ the equality $x_1=x_2$ implies $y_1=y_2$;
\item {\em strictly increasing} if for any $\langle x_1,y_1\rangle,\langle x_2,y_2\rangle\in f$ the strict inequality $x_1<x_2$ implies $y_1<y_2$;
\item {\em strictly decreasing} if for any $\langle x_1,y_1\rangle,\langle x_2,y_2\rangle\in f$ the inequality $x_1< x_2$ implies $y_1>y_2$;
\item {\em strictly monotone} if $f$ is strictly increasing or strictly decreasing. 
\end{enumerate}

\begin{lemma}\label{l:i} Each strictly increasing function $f\subset\II^2$ is a subset of a Borel strictly increasing function $\bar f\subset \II^2$.
\end{lemma}

\begin{proof} It follows that the strictly increasing function $f$ is a strictly increasing bijective function between the sets $\pr_1[f]=\{x\in\II:\exists y\in\II\; \;\langle x,y\rangle\in f\}$ and $\pr_2[f]=\{y\in \II:\exists x\in\II\;\;\langle x,y\rangle\in f\}$. 
It is well-known that monotone functions of one real variable have at most countably many discontinuity points. Consequently, the sets of discontinuity points of  the strictly monotone functions $f$ and $f^{-1}$ are at most countable. This allows us to find a countable set $D_f\subset f$ such that the set $f\setminus D_f$ coincides with the graph of some increasing homeomorphism between subsets of $\II$. Replacing $D_f$ by a larger countable set, we can assume that $D_f=f\cap(\pr_1[D_f]\times\pr_2[D_f])$, where $\pr_1,\pr_2:\II^2\to\II$ are coordinate projections. By the Lavrentiev Theorem \cite[3.9]{Kechris}, the homeomorphism $f\setminus D_f$ extends to a (strictly increasing) homeomorphism $h\subset \II^2$ between $G_\delta$-subsets of $\II^2$ such that $f\setminus D_f$ is dense in $h$. It is easy to check that the Borel subset $\bar f=(h\setminus (\pr_1[D_f]\times\pr_2[D_f])\cup D_f$ is a strictly increasing function extending $f$.
\end{proof}

By analogy we can prove

\begin{lemma}\label{l:d} Each strictly decreasing function $f\subset\II^2$ is a subset of a Borel strictly decreasing function $\bar f\subset \II^2$.
\end{lemma}

Now we are ready to prove the main result of this section.

\begin{lemma}\label{l:main} Under $\w_1=\mathfrak c$ the multifunction $P^{-1}:\II^2\multimap \ddot\II^2$ has an $F_\sigma$-measurable selection.
\end{lemma}

\begin{proof} Let $\M$ be the set of infinite strictly monotone Borel functions $f\subset\II^2$. Since $\w_1=\mathfrak c$, the set $\M$ can be written as $\M=\{f_\alpha\}_{\alpha<\w_1}$.  It is clear $\bigcup_{\alpha<\w_1}f_\alpha=\II^2$. So, for any point $z\in\II^2$ we can find the smallest ordinal $\alpha_z<\w_1$ such that $z\in f_{\alpha_z}$. Consider the sets 
$$
\begin{aligned}
\mathrm{L}&:=\{z\in \II^2:f_{\alpha_z}\mbox{ is strictly increasing}\}\mbox{ \ and \ }\\\Gamma&:=\{z\in\II^2:\mbox{$f_{\alpha_z}$ is strictly decresing}\}=\II^2\setminus \mathrm L.
\end{aligned}$$ 
Define a selection $s:\II^2\to \ddot\II^2$ of the multimap $P^{-1}:\II^2\multimap \ddot\II^2$ letting $$s(\langle x,y\rangle)=\begin{cases}\langle x_1,y_1\rangle&\mbox{if $\langle x,y\rangle\in \mathrm L$},\\
\langle x_1,y_0\rangle&\mbox{if $\langle x,y\rangle\in \Gamma$},
\end{cases}
$$for $\langle x,y\rangle\in\II^2$.

We claim that the  function $s:\II^2\to \ddot\II^2$ is $F_\sigma$-measurable.
Given any open set $U\subset \ddot\II^2$, we should prove that its preimage $s^{-1}[U]$ of type $F_\sigma$ in $\II^2$. Consider the open subset $V:=U\cap (0_1,1_0)^2\subset \ddot\II^2$ of $U$. Using Lemma~\ref{l:I}, it can be shown that the set $s^{-1}[U\setminus V]\subset \II^2\setminus (0,1)^2$ is of type $F_\sigma$ in $\II^2$.
Therefore, it remains to show that the preimage $s^{-1}[V]$ is of type $F_\sigma$ in $\II^2$.

Let $\IQ:=\{\frac{n}{m}:n,m\in\mathbb N,\;n<m\}$ be the set of rational numbers in the interval $(0,1)$.

Consider the subsets $\mathrm L_V:=\mathrm L\cap s^{-1}(V)$ and $\Gamma_V:=\Gamma\cap s^{-1}(V)$. For every $\langle x,y\rangle\in  \mathrm L_V$ we have $s(\langle x,y\rangle)=\langle x_1,y_1\rangle\in V$ and by the definition of the topology of the split interval, we can find rational numbers $a(x,y),b(x,y)\in \IQ$ such that $x<a(x,y)$, $y<b(x,y)$ and $s(\langle x,y\rangle)=\langle x_1,y_1\rangle\in \big[x_1,a(x,y)_0\big)\times \big[y_1,b(x,y)_0\big)\subset V$. 
Then $$\big[x,a(x,y)\big)\times\big[y,b(x,y)\big)=s^{-1}\big[[x_1,a(x,y)_0)\times[y_1,b(x,y)_0)\big]\subset s^{-1}[V].$$

On the other hand, for every $\langle x,y\rangle\in \Gamma_V$ there are rational numbers $a(x,y),b(x,y)\in\IQ$ such that $x<a(x,y)$, $b(x,y)<y$ and $s(\langle x,y\rangle)=\langle x_1,y_0\rangle=\big[x_1,a(x,y)_0\big)\times \big(b(x,y)_1,y_0\big]\subset V$. In this case
 $$\big[x,a(x,y)\big)\times\big(b(x,y),y\big]=s^{-1}\big[[x_1,a(x,y)_0)\times(b(x,y)_1,y_0]\big)\subset s^{-1}[V].$$

It follows that
$$
s^{-1}[V]=\Big(\bigcup_{\langle x,y\rangle\in \mathrm L_V}\big[x,a(x,y)\big)\times\big[y,b(x,y)\big)\Big)\cup\Big(\bigcup_{\langle x,y\rangle\in \Gamma_V}\big[x,a(x,y)\big)\times\big(b(x,y),y\big]\Big).
$$

This equality and the following claim imply that the set $s^{-1}[V]$ is of type $F_\sigma$ in $\II^2$.

\begin{claim}\label{cl:main} There are countable subsets $\mathrm L'\subset\mathrm L_V$ and $\Gamma'\subset \Gamma_V$ such that
$$\bigcup_{\langle x,y\rangle\in \mathrm L_V} \big[x,a(x,y))\times[y,b(x,y)\big)=\bigcup_{\langle x,y\rangle\in \mathrm L'} \big[x,a(x,y)\big)\times\big[y,b(x,y)\big)$$and
$$\bigcup_{\langle x,y\rangle\in \Gamma_V} \big[x,a(x,y)\big)\times\big(b(x,y),y\big]=\bigcup_{\langle x,y\rangle\in \Gamma'} \big[x,a(x,y)\big)\times\big(b(x,y),y\big].$$
\end{claim}

We shall show how to find the countable set $\mathrm L'\subset\mathrm L_V$. The countable set $\Gamma'\subset\Gamma_V$ can be found by analogy.

 For rational numbers $r,q\in\IQ$, consider the set $$\mathrm L_{r,q}=\{\langle x,y\rangle\in \mathrm L_V:a(x,y)=r,\;b(x,y)=q\}$$and observe that $\mathrm L_V=\bigcup_{r,q\in\IQ}\mathrm L_{p,q}.$
 
 \begin{claim}\label{cl4} For any rational numbers $r,q\in\IQ$ there exists a countable subset $\mathrm L_{r,q}'\subset\mathrm L_{r,q}$ such that $$\bigcup_{\langle x,y\rangle\in\mathrm L'_{r,q}}[x,r)\times [y,q)=\bigcup_{\langle x,y\rangle\in\mathrm L_{r,q}}[x,r)\times [y,q).$$
 \end{claim}

\begin{proof} For every rational numbers $r'\le r$ and $q'\le q$, consider the numbers $$
\underline{y}(r'):=\inf\{y:\langle x,y\rangle\in \mathrm L_{r,q},\;x<r'\}\mbox{ and }
\underline{x}(q'):=\inf\{x:\langle x,y\rangle\in \mathrm L_{r,q},\;y<q'\}.$$
Choose countable subsets $\mathrm L_{r,q}^{r'\!,0},L_{r,q}^{0,q'}\subset \mathrm L_{r,q}$ such that 
$$\underline{y}(r')=\inf\{y:\langle x,y\rangle\in \mathrm L^{r'\!,0}_{r,q},\;x<r'\}\mbox{ and }
\underline{x}(q')=\inf\{x:\langle x,y\rangle\in \mathrm L^{0,q'}_{r,q},\;y<q'\}
$$
and moreover,
$$\underline{y}(r')=\min\{y:\langle x,y\rangle\in \mathrm L^{r'\!,0}_{r,q},\;x<r'\}\mbox{ \ if \ }\underline{y}(r')=\min\{y:\langle x,y\rangle\in \mathrm L_{r,q},\;x<r'\}.
$$
and
$$
\underline{x}(q')=\min\{x:\langle x,y\rangle\in \mathrm L^{0,q'}_{r,q},\;y<q'\}
\mbox{ \ if \ }
\underline{x}(q')=\min\{x:\langle x,y\rangle\in \mathrm L_{r,q},\;y<q'\}.$$

Consider the countable subset
 $$\mathrm L_{r,q}'':=\textstyle\bigcup\{\mathrm L^{r'\!,0}_{r,q}\cup\mathrm L_{r,1}^{0,q'}:r',q'\in \IQ,\; r'<r,\;q'<q\}$$
 of $\mathrm L_{r,q}$.
 
 \begin{claim}\label{cl5}
$\displaystyle\bigcup\limits_{\langle x,y\rangle\in L_{r,q}}\big([x,r)\times [y,q)\big)\setminus \{\langle x,y\rangle\}\subset \bigcup\limits_{\langle x,y\rangle\in \mathrm L_{r,q}''}[x,r)\times[y,q).$
\end{claim}

\begin{proof}
Fix any pairs $\langle x,y\rangle\in \mathrm L_{r,q}$ and  $\langle x',y'\rangle \in \big([x,r)\times [y,q)\big)\setminus \{\langle x,y\rangle\}$. Three cases are possible:
\begin{enumerate}
\item $x<x'<r$ and $y<y'<q$;
\item $x=x'$ and $y<y'<q$;
\item $x<x'<r$ and $y=y'$.
\end{enumerate}
In the first case there exist rational numbers $r',q'$ such that $x<r'<x'<r$ and $y<q'<y'<q$. The definition of $\underline{x}(q')$ ensures that $\underline{x}(q')\le x<x'$. By the choice of the family $\mathrm L_{r,q}^{0,q'}$, there exists $\langle x'',y''\rangle \in\mathrm L_{r,q}^{0,q'}\subset\mathrm L_{r,q}''$ such that $x''<x'<r$ and $y''<q'<y'<q$. Then $\langle x',y'\rangle\in [x'',r)\times[y'',q)$.

Next, assume that $x=x'$ and $y< y'<q$. In this case we can choose a rational number $q'$ such that $y<q'<y'$. It follows that $\underline{x}(q')\le x=x'$. If $\underline{x}(q')<x'$, then by the definition of the family $\mathrm L_{r,q}^{0,q'}$, there exists $\langle x'',y''\rangle \in\mathrm L_{r,q}^{0,q'}\subset \mathrm L_{r,q}$ such that $x''<x'<r$ and $y''<q'<y'<q$. Then $\langle x',y'\rangle\in [x'',r)\times[y'',q)$. 

So, we assume that $\underline{x}(q')=x'=x$ and hence $\underline{x}(q')=x=\min\{x'':\langle x'',y''\rangle\in \mathrm L_{p,q}:y''<q'\}$. In this case $x'=\underline{x}(q')=x''$ for some $\langle x'',y''\rangle\in\mathrm L_{r,q}^{0,q'}\subset\mathrm L_{r,q}''$ with $y''<q'<y'<q$. Then $\langle x',y'\rangle\in [x'',r)\times [y'',q)$.

By analogy, in the third case ($x<x'<r$ and $y=y''$) we can find a pair $\langle x'',y''\rangle\in L_{r,q}''$ such that $\langle x',y'\rangle\in [x'',r)\times[y'',q)$.
\end{proof}
Claim~\ref{cl5} implies that the set 
$$D_{r,q}=\Big(\bigcup_{\langle x,y\rangle \in \mathrm L_{r,q}}[x,r)\times [y,q)\Big)\setminus \Big(
\bigcup_{\langle x,y\rangle\in \mathrm L_{r,q}''}[x,r)\times[y,q)\Big)$$is contained in $\mathrm L_{r,q}.$
 \end{proof}

\begin{claim} The set $D_{r,q}$ is a strictly decreasing function.
\end{claim}

\begin{proof} First we show that $D_{r,q}$ is a function. Assuming that $D_{r,q}$ is not a function, we can find two pairs $\langle x,y\rangle,\langle x,y'\rangle\in D_{r,q}$ with $y<y'$. Applying Claim~\ref{cl5}, we conclude that $$\langle x,y'\rangle\in \big([x,r)\times[y,q)\big)\setminus\{\langle x,y\rangle\}\subset \bigcup_{\langle x'',y''\rangle\in L_{r,q}''}[x'',r)\times[y'',q)$$and hence $\langle x,y'\rangle\notin D_{r,q}$, which contradicts the choice of the pair $\langle x,y'\rangle$. This contradiction shows that $D_{r,q}$ is a function.

Assuming that $D_{r,q}$ is not strictly decreasing, we can find pairs $\langle x,y\rangle,\langle x',y'\rangle\in D_{r,q}$ such that $x<x'$ and $y\le y'$. 
Applying Claim~\ref{cl5}, we conclude that $$\langle x',y'\rangle\in \big([x,r)\times[y,q)\big)\setminus\{\langle x,y\rangle\}\subset \bigcup_{\langle x'',y''\rangle\in L_{r,q}''}[x'',r)\times[y'',q)$$and hence $\langle x',y'\rangle\notin D_{r,q}$, which contradicts the choice of the pair $\langle x,y'\rangle$. This contradiction shows that $D_{r,q}$ is strictly decreasing.
\end{proof}

\begin{claim}\label{cl6} The set $D_{r,q}$ is at most countable.
\end{claim}

\begin{proof} To derive a contradiction, assume that $D_{r,q}$ is uncountable.
By Lemma~\ref{l:d}, the strictly decreasing function $D_{r,q}$ is contained in some Borel strictly decreasing function, which is equal to $f_\alpha$ for some ordinal $\alpha<\w_1$. Since the intersection of a strictly increasing function and a strictly decreasing function contains at most one point, the set $$D_{r,q}'=\textstyle\bigcup\{D_{r,q}\cap f_\beta:\beta\le \alpha,\;f_\beta\mbox{ is strictly increasing}\}$$ is at most countable. We claim that $D_{r,q}=D'_{r,q}$. To derive a contradiction, assume that $D_{r,q}\setminus D'_{r,q}$ contains some pair $z=\langle x,y\rangle$. It follows from $z\in D_{r,q}\subset f_\alpha$ that $\alpha_z\le\alpha$.  Since $z\notin D_{r,q}'$, the strictly monotone function $f_{\alpha_z}\ni z$ is not strictly increasing and hence $f_{\alpha_z}$ is strictly decreasing. Then the definition of the set $\mathrm L$ guarantees that $z\notin \mathrm L$, which contradicts the inclusion $z\in D_{r,q}\subset \mathrm L_{r,q}\subset \mathrm L$. 
\end{proof}

Now consider the countable subset $\mathrm L'_{r,q}:=\mathrm L_{r,q}''\cup D_{r,q}$ of $\mathrm L_{r,q}$ and observe that
$$\displaystyle\bigcup\limits_{\langle x,y\rangle\in L_{r,q}}\big([x,r)\times [y,q)\big)\subset \bigcup\limits_{\langle x,y\rangle\in \mathrm L_{r,q}''}[x,r)\times[y,q).$$
This completes the proof of Claim~\ref{cl4}.

\begin{claim}\label{cl8} There exists a countable subset $\mathrm L'\subset \mathrm L_V$ such that
$$\bigcup_{\langle x,y\rangle\in \mathrm L'}\big([x,a(x,y))\times [y,b(x,y)\big)= 
\bigcup_{\langle x,y\rangle\in \mathrm L_V}\big([x,a(x,y))\times [y,b(x,y)\big).$$
\end{claim}

\begin{proof} By Claim~\ref{cl4}, for any rational numbers $r,q\in\IQ$ there exists a countable subset $\mathrm L_{r,q}'\subset \mathrm L_{r,q}$ such that 
\begin{multline*}
\bigcup_{\langle x,y\rangle\in \mathrm L'_{r,q}}\big([x,a(x,y))\times [y,b(a,y))\big)=\bigcup_{\langle x,y\rangle\in \mathrm L'_{r,q}}\big([x,r)\times [y,q)\big)=\\
=\bigcup\limits_{\langle x,y\rangle\in\mathrm  L_{r,q}}\big([x,r)\times [y,q)\big)=
\bigcup\limits_{\langle x,y\rangle\in\mathrm L_{r,q}}\big([x,a(x,y))\times [y,b(x,y))\big).
\end{multline*}
Since $\mathrm L_V=\bigcup_{r,q\in\IQ}\mathrm L_{r,q}$, the countable set $\mathrm L':=\bigcup_{r,q\in\IQ}\mathrm L'_{r,q}$ has the required property.
\end{proof} 

By analogy with Claim~\ref{cl8} we can prove

\begin{claim}\label{cl10} There exists a countable subset $\Gamma'\subset \Gamma_V$ such that
$$\bigcup_{\langle x,y\rangle\in \Gamma'}\big([x,a(x,y))\times (b(x,y),y]\big)= 
\bigcup_{\langle x,y\rangle\in \Gamma_V}\big([x,a(x,y))\times (b(x,y),y]\big).$$
\end{claim}
Claims~\ref{cl8} and \ref{cl10} complete the proof of Claim~\ref{cl:main} and the proof of Lemma~\ref{l:main}.
\end{proof}

\section{Acknowledgement}

The author expresses his sincere thanks to the Mathoverflow user @YCor (Yves Cornulier) for the idea of the proof of Lemma~\ref{l1} and to Du\v san Repov\v s for valuable information on measurable selectors.

\end{document}